\documentclass[11pt]{amsart}
\usepackage{amsmath}
\usepackage{amsfonts}
\usepackage{float}
\usepackage{mathrsfs}
\usepackage{mathtools}
\usepackage{latexsym}
\usepackage{amsmath,amsthm,amssymb, amscd,color}
\usepackage{mathrsfs}
\usepackage[all]{xy}
\usepackage[all]{xy}
\usepackage{graphicx}
\usepackage{resizegather}
\usepackage{tikz}
\usepackage{tikz-cd}
\usepackage{breqn}
\usepackage{cleveref}
\usepackage{enumerate}

\usepackage{blkarray} 
\usepackage{framed}

\newtheorem{innercustomgeneric}{\customgenericname}
\providecommand{\customgenericname}{}
\newcommand{\newcustomtheorem}[2]{%
	\newenvironment{#1}[1]
	{%
		\renewcommand\customgenericname{#2}%
		\renewcommand\theinnercustomgeneric{##1}%
		\innercustomgeneric
	}
	{\endinnercustomgeneric}
}

\newcustomtheorem{customthm}{Theorem}
\newcustomtheorem{customlemma}{Lemma}
\newcustomtheorem{customprop}{Proposition}

\theoremstyle{plain}
\newtheorem{theorem}{Theorem}[section]
\newtheorem{proposition}[theorem]{Proposition}
\newtheorem{lemma}[theorem]{Lemma}
\newtheorem{corollary}[theorem]{Corollary}
\numberwithin{equation}{section}

\theoremstyle{definition}
\newtheorem{definition}[theorem]{Definition}
\newtheorem{example}[theorem]{Example}

\newtheorem{remark}[theorem]{Remark}

\theoremstyle{remark}

\newcommand{\RR}{\mathbb{R}}

\newcommand{\NN}{\mathbb{N}}

\newcommand{\SN}{\mathbb{S}}

\begin{document}

\title[Darboux first integrals of Kolmogorov systems with invariant $n$-sphere]{Darboux first integrals of Kolmogorov systems with invariant $n$-sphere}

\author[S. Jana]{Supriyo Jana}
\address{Department of Mathematics, Indian Institute of Technology Madras, India}
\email{supriyojanawb@gmail.com}

\author[S. Sarkar]{Soumen Sarkar}
\address{Department of Mathematics, Indian Institute of Technology Madras, India}
\email{soumen@iitm.ac.in}

\date{\today}
\thanks{}

\abstract 
In this paper, we characterize all polynomial Kolmogorov vector fields for which the standard $n$-sphere is invariant. We exhibit completely integrable Kolmogorov vector fields of degree $m$ on $\SN^n$ for any $m >2$. Then, we show that there is no cubic Hamiltonian Kolmogorov vector field that makes an odd-dimensional sphere invariant.  We examine the conditions under which a cubic Kolmogorov vector field has a Darboux first integral. In many cases, we determine whether they constitute necessary and sufficient conditions. Moreover, we study the complete integrability of cubic Kolmogorov vector fields having an invariant $n$-sphere.
\endabstract

\maketitle

\section{Introduction}
Let $R_1,\ldots,R_d$ be polynomials in $\mathbb{R}[x_1,\ldots,x_d]$. Then, the following system of differential equations
\begin{equation} \label{eq: I1}
\frac{{\rm d}x_i}{{\rm d t}} = R_i(x_1, \ldots, x_d),~i = 1,\ldots, d,
\end{equation}
is called a polynomial differential system in $\mathbb{R}^d$. The differential operator
\begin{equation}  \label{eq: I2}
 \chi = \sum_{i=1}^d R_i \frac{\partial}{\partial x_i}    
\end{equation}
is called the vector field associated with the system \eqref{eq: I1}. The number $\max\limits_{1\leq i\leq d} \{ \deg R_i\}$ is called the degree of the polynomial vector field \eqref{eq: I2}.

If $R_1, \ldots, R_d$ are homogeneous polynomials of the same degree, then  $\chi$ is called a homogeneous vector field.

 The system \eqref{eq: I1} is called a polynomial Kolmogorov system in $\RR^d$ if $R_i=x_i Q_i$ for some $Q_1, \ldots, Q_d \in \RR[x_1, \ldots, x_d]$. The associated vector field is called a polynomial Kolmogorov vector field. Some classes of Kolmogorov systems in $\RR^2$ and $\RR^4$ were studied in \cite{tigan2021analysis} and \cite{LlXi17}, respectively. The integrability of a class of Kolmogorov systems in $\RR^n$ has been explored in \cite{LlRaRa20}. The authors and Benny \cite{BJS_24} discussed the phase portraits of cubic Kolmogorov vector fields on $\SN^2$. We note that Kolmogorov systems \cite{kol36} have been found as models of various phenomena in plasma physics \cite{LMQ22}, economics \cite{AFSS94}, and other areas, including ecology. We note that a quadratic Kolmogorov system in $\RR^d$ is known as a Lotka--Volterra system \cite{LlRaRa20, lotka1925}. The latter system has been used to understand the dynamics of predator--prey interactions in ecological systems \cite{AFSS94, saez1999dynamics} and in disease spread \cite{ghasemabadi2021investigating}.

A subset of $\RR^d$ is called invariant for the vector field \eqref{eq: I2} if the set is a union of some trajectories of the vector field.  The zero locus of an irreducible polynomial $f \in \mathbb{R}[x_1, \ldots,x_d]$  is called an irreducible algebraic set of $f$. An irreducible algebraic set of $f$ is invariant if there exists $K\in \RR[x_1, \ldots,x_d]$ such that $\chi f = Kf$. In this setting, $\chi$ is called a polynomial vector field on the hypersurface $\{f=0\}$ and $K$ is referred to as the cofactor.

We note that the study of polynomial vector fields started blooming in the 1970s; see, for instance, \cite{CL90, Go69, Ho79, Jou79}. One can ask for the number of possible invariant algebraic sets for a given vector field, see \cite{bolanos2011number, bolanos2013number, Sukulski1996}.

 Darboux \cite{Dar78} showed the existence of a (rational) first integral (see Definition \ref{def_firs_intg} and the paragraph after that) when the polynomial vector field in $\RR^2$ possesses a sufficient number of invariant algebraic curves. Jouanolou \cite{Jou79} extended similar results in higher dimensions. There are $(d-1)$ functionally independent local first integrals of any vector field in $\RR^d$ around non-singular points, see \cite[Chapter 10] {arnold1992ordinary}. Thus, one can ask whether a polynomial system in $\RR^d$ has a polynomial or a rational first integral. Some more developments on this topic can be found in \cite{CL99,llibre2018darboux, LliZh09}. However, several interesting dynamical properties of polynomial vector fields, such as the Kolmogorov vector fields in $\mathbb{R}^n$, remain to be explored. The existence of enough functionally independent first integrals means that the system is completely integrable, see Definition \ref{def_comp_intg}. Llibre and Bola\~{n}os \cite{LliBol11} showed that polynomial vector fields on a regular algebraic hypersurface having a sufficient number of invariant algebraic hypersurfaces possess a rational first integral. The Darboux theory of integrability for polynomial vector fields on $\SN^n$ has been studied in \cite{llibre2018darboux}. Therefore, it is essential to study various invariant algebraic hypersurfaces for a Kolmogorov vector field on $\SN^n$ and subsequently study their Darboux first integrals.

The paper is organized as follows. Section \ref{sec:basic_on_vf} recalls the definitions of first integral, complete integrability, Hamiltonian vector field, and some basic notions required in this paper.

In Section \ref{sec:vf_on_sn}, first, we prove the following characterization. 
\begin{theorem}\label{thm:kolm-vfld}
    Let $\chi=(P_1, \ldots,P_{n+1})$ be a degree $m$ polynomial Kolmogorov vector field in $\RR^{n+1}$. Then $\chi$ is a vector field on $\SN^n$ if and only if
    \begin{equation}\label{eq:deg-m-form}
            P_i=x_i\Big(\Big(1-\sum\limits_{k=1}^{n+1} x_k^2\Big)\Tilde{f}_i+\sum\limits_{j=1}^{n+1}\Tilde{A}_{ij}x_j^2\Big),~ 1\leq i\leq n+1,
    \end{equation}
where $\Tilde{f}_i, \Tilde{A}_{ij}$ are polynomials with degree less than or equal to $(m-3)$ such that $\Tilde{A}:=(\Tilde{A}_{ij})$ is a skew-symmetric matrix.
\end{theorem}
Then, we give a sufficient condition for the existence of a Darboux first integral of a cubic Kolmogorov vector field, see \Cref{thm_cubic_firint}. We construct completely integrable Kolmogorov vector fields of arbitrary degree on $\SN^n$, see \Cref{thm:cmplt-int-vfld}. Next, we prove an important observation.
\begin{theorem}\label{thm_no_Hamilton}
There exists no cubic Kolmogorov vector field in $\RR^{2n}$ that is Hamiltonian and makes $\SN^{2n-1}$ invariant.
\end{theorem}
We study invariant spheres of cubic and quartic homogeneous Kolmogorov vector fields. 
\begin{theorem}\label{thm_inv_grt_deg1}
   Suppose that $\chi$ is a cubic homogeneous Kolmogorov vector field on $\SN^n$ of the form \eqref{eq:deg-m-form}. If $\sum\limits_{i=1}^{n+1}a_ix_i=0$ is an invariant hyperplane of $\chi$ with cofactor $K$, then the following holds.
   \begin{enumerate}[(i)]
       \item $\sum\limits_{i,j=1}^{n+1} a_i\Tilde{A}_{ij}=(a_1+\cdots+a_{n+1}) \times$ (sum of all the coefficients in $K$).
       \item $K(a_1,\ldots,a_{n+1})=0$.
   \end{enumerate}
\end{theorem}
 
In Section \ref{sec:inv_mer_par}, first, we characterize the invariant hyperplanes of cubic Kolmogorov vector fields on $\SN^n$ with cofactors of the form $k_0+\sum\limits_{i=1}^{n+1}k_ix_i^2$, see \Cref{thm_hypln_cofac}. We provide a necessary and sufficient condition when the function $H=g_{n+2}^{\beta_{n+2}}\prod\limits_{i=1}^{n+1}x_i^{\beta_i}$ is a first integral for certain cubic Kolmogorov vector fields on $\SN^n$, see \Cref{first-integral-ns}.  We show that if a cubic Kolmogorov vector field has two invariant spheres with the centers at origin, then the vector field is homogeneous and the defining polynomials of the spheres are first integrals, see \Cref{prop_fint1}. Finally, we prove the following. 
\begin{theorem}\label{thm_indp_fi}
Let $\chi=(P_1,\ldots,P_{n+1})$ be a cubic Kolmogorov vector field  on $\SN^n$ with
    $$P_i=x_i\Big(\alpha_i\Big(1-\sum_{k=1}^{n+1}x_k^2\Big)+\sum\limits_{j=1}^{n+1}\Tilde{A}_{ij}x_j^2\Big)$$ having an invariant hypersurface $g_{n+2}=0$ with cofactor $K_{n+2}=k_0+\sum\limits_{i=1}^{n+1}k_ix_i^2$ where
    $\alpha_i, k_i, k_0\in \RR$ for $i=1,\ldots,n+1$ and $A=(\Tilde{A}_{ij})$ is a constant skew-symmetric matrix. Suppose that there exist $z_{i1},\ldots,z_{i(n+1)}\in \RR^{n+1}$ with non-zero coordinates such that $\frac{\partial g_{n+2}}{\partial x_i}(z_{ij})$ and $g_{n+2}(z_{ij})$ are non-zero and the vectors $$\begin{pmatrix}
        x_1\frac{\partial g_{n+2}}{\partial x_1} & \ldots&x_{i-1}\frac{\partial g_{n+2}}{\partial x_{i-1}}&x_{i+1}\frac{\partial g_{n+2}}{\partial x_{i+1}}&\ldots&x_{n+1}\frac{\partial g_{n+2}}{\partial x_{n+1}}&-g_{n+2}
    \end{pmatrix}$$ evaluated at $z_{i1},\ldots, z_{i(n+1)}$ are independent where $i,j \in \{1,\ldots,n+1\}$. Then, the vector field has functionally independent first integrals $H_i:=g_{n+2}^{y_{i(n+2)}}\prod_{j=1}^{n+1}x_j^{y_{ij}}$ for $i\in \{1,\ldots,n\}$ if and only if ${\rm rank}(B)\leq 2$, where $B$ is given by \eqref{eq:matrix-b}.
\end{theorem}

\section{Preliminaries}\label{sec:basic_on_vf}
In this section, we recall the definitions of first integral, complete integrability, Hamiltonian vector field, syzygies of some polynomials, and some basic notions required in this paper.

\begin{definition}\label{def_firs_intg}
Let $U$ be an open subset of $\RR^d$. A non-constant differentiable map $H \colon U \to \mathbb{R}$ is said to be a first integral
of the vector field \eqref{eq: I2} on $U$ if each solution curve of the system \eqref{eq: I1} contained in $U$ lies on one of the level surface of $H$; i.e., $H(x_1(t),\ldots,x_d(t)) =$ const for all values of $t$ for which the solution $(x_1(t),\ldots,x_d(t))$ is defined and contained in $U$.
\end{definition}
Note that $H$ is a first integral of the vector field \eqref{eq: I2} on $U$ if and only if $\chi H=0$ on $U$. When $H$ is a rational function, it is referred to as a rational first integral. The existence of a first integral reduces the analysis of the system by one dimension. Moreover, if the vector field \eqref{eq: I2} admits a first integral, then it possesses infinitely many invariant algebraic sets. A proof of this result can be found on page 102 of \cite{Jou79}.

A restricted version of the Darboux theory of integrability is the following.
\begin{proposition}
    Suppose that a polynomial vector field in $\RR^d$ has $p$ invariant algebraic hypersurfaces $f_i=0$ with cofactor $K_i$ for $i=1,\ldots,p$. If there exist $\lambda_1, \ldots, \lambda_p \in \RR$, not all zero, such that $\sum\limits_{i=1}^p \lambda_i K_i=0$, then $f_1^{\lambda_1}\ldots f_p^{\lambda_p}$ is a first integral of the vector field.
\end{proposition}
 This type of first integral is known as a Darboux first integral. More general results using independent singular points and exponential factors can be seen in \cite{Jou79, llibre2000invariant, weil1995constantes}. A set of $m$ functions is called functionally independent in an open set $U$ if the rank of the derivative of the mapping $f\colon U\to \RR^m$ defined by these functions is $m$ at some point of $U$.

\begin{definition}\label{def_comp_intg}
The vector field \eqref{eq: I2} is said to be completely integrable in an open set $U$ if it admits $(d-1)$ functionally independent first integrals defined on $U$.
\end{definition}
If the vector field \eqref{eq: I2} is completely integrable with $(d-1)$ functionally independent first integrals $H_1,\ldots,H_{d-1}$ then the trajectories of the system are contained in $\bigcap\limits_{i=1}^{d-1}\{H_i=c_i\}$ for some $c_i\in \RR$. A discussion of the complete integrability of vector fields on $\RR^d$ can be found in \cite{LlRaRa20}.
Considering a permutation of coordinates, Definition 1 in \cite[Section 2.14]{Perko} can be rewritten as follows.
\begin{definition}
    A vector field $\chi=(P_1,\ldots, P_{2n})$ is called a Hamiltonian vector field if there exists a differentiable function $H \colon \RR^{2n}\to \RR$ such that $P_{2i-1}=-\frac{\partial H}{\partial x_{2i}}$ and $P_{2i}=\frac{\partial H}{\partial x_{2i-1}}$ for $1\leq i\leq n$.
\end{definition}

The function $H$ is called a Hamiltonian for the vector field. It follows that $H$ is a global first integral. So, the function $H$ is conserved along any trajectories of the vector field. We note that \cite{LlXi17} considers this definition to study Hamiltonian vector fields on $\RR^4_{+}$.

\begin{definition}
    Let $R$ be a ring and $M$ be an $R$-module. Suppose that $(g_1,\ldots,g_m)\in M^m$. A syzygy of $(g_1,\ldots,g_m)$ is a tuple $(f_1,\ldots,f_m)\in R^m$ such that $\sum\limits_{i=1}^m f_ig_i=0$.
\end{definition}
The set of all syzygies of $(g_1,\ldots,g_m)$ forms an $R$-module which is called the
(first) syzygy module of $(g_1,\ldots,g_m)$ and is denoted by ${\rm Syz}_R(g_1,\ldots,g_m)$. For more information on syzygies, one can see \cite{MR2723052}. In this paper, we consider $R=M=\RR[x_1,\ldots,x_d]$, the polynomial ring in $d$ variables over real.

The following result is algebraic. However, it is crucial to classify polynomial Kolmogorov vector fields on $\SN^n$.  
\begin{lemma}[\cite{jana2024dynamics}]\label{lem:sum-pi-xi-0}
Let $(Q_1,\ldots,Q_d)$ be a syzygy of $(x_1^k,\ldots,x_d^k)$ in $\RR[x_1,\ldots,x_d]$ for some $k\geq 1$.  Then $Q_i=\sum\limits_{j=1}^d  \Tilde{A}_{ij}x_j^k$ for some $\Tilde{A}_{ij}\in \RR[x_1,\dots,x_d]$ such that $A=(\Tilde{A}_{ij})_{d\times d}$ is a skew-symmetric matrix.
\end{lemma}

The intersection of a hyperplane $L:=\Big\{\sum\limits_{i=1}^{n+1}a_ix_i+b=0\Big\}$ with $\SN^n$ is said to be an $(n-1)$-sphere in $\SN^n$ where $a_i,b\in \RR$ with $a_1^2 + \cdots + a_{n+1}^2=1$ and $|b|<1$. If the hyperplane $L$ passes through the origin (i.e., $b=0$ in $L$), the intersection is called a great $(n-1)$-sphere in $\SN^n$.


\section{Polynomial Kolmogorov vector fields on $\SN^n$}\label{sec:vf_on_sn}
In this section, we characterize any polynomial Kolmogorov vector fields on the standard $n$-sphere in $\RR^{n+1}$. Then, we discuss their first integrals and complete integrability. We show that there is no cubic Hamiltonian Kolmogorov vector field on odd-dimensional spheres. We study invariant spheres of cubic and quartic homogeneous Kolmogorov vector fields on $\SN^n$. We note that the paper \cite{zerz2012controlled} gave a necessary and sufficient condition when a hypersurface is invariant of a polynomial vector field. However, the paper \cite{jana2024dynamics} presented a more explicit form of a polynomial vector field with an invariant sphere. We use the latter result for our convenience.

\begin{proof}[\textbf{Proof of  \Cref{thm:kolm-vfld}}]
  Suppose that $\chi=(P_1,\ldots,P_{n+1})$ is a degree $m$ Kolmogorov vector field on $\mathbb{S}^n$. By \cite[Remark 3.2]{jana2024dynamics}, we have
 \begin{equation}\label{eq:kolm}
   P_i=\Big(1-\sum\limits_{k=1}^{n+1} x_k^2\Big) f_{i} +\sum\limits_{j=1}^{n+1} A_{ij}x_j,
    \end{equation}
    for some polynomial $f_i,A_{ij} \in \RR[x_1, \ldots, x_{n+1}]$ with $\deg f_i\leq m-2$ and $\deg A_{ij}\leq m-1$ such that $(A_{ij})_{(n+1)\times (n+1)}$ is a skew-symmetric matrix. Moreover, $x_i$ divides $P_i$, for $i=1,\ldots,n+1$. Therefore, $P_i=\Tilde{P}_ix_i$ for some $\Tilde{P}_i\in \RR[x_1,\ldots,x_{n+1}]$, for $i=1,\ldots,n+1$. Observe that $$\sum\limits_{i=1}^{n+1}\Tilde{P}_ix_i^2=\Big(1-\sum\limits_{k=1}^{n+1} x_k^2\Big) \Big(\sum\limits_{i=1}^{n+1}f_i x_i\Big).$$
This simplifies to $\sum\limits_{i=1}^{n+1}\Tilde{P}_ix_i^2 + \Big(\sum\limits_{i=1}^{n+1}f_i x_i\Big) \sum\limits_{k=1}^{n+1}  x_k^2 = \sum\limits_{i=1}^{n+1}f_i x_i.$
Hence, $\sum\limits_{i=1}^{n+1}\Tilde{f}_i x_i^2 =\sum\limits_{i=1}^{n+1}f_i x_i$ with $\deg \Tilde{f}_i\leq m-3$, leading to $\sum\limits_{i=1}^{n+1}(f_i-\Tilde{f}_ix_i)x_i=0$. By \Cref{lem:sum-pi-xi-0}, we have $f_i=\Tilde{f}_ix_i+\sum\limits_{j=1}^{n+1} B_{ij}x_j$ for some $B_{ij}\in \RR[x_1,\ldots,x_{n+1}]$ with $\deg B_{ij}\leq m-3$ such that $(B_{ij})_{(n+1)\times (n+1)}$ is a skew-symmetric matrix. Substituting $f_i$ in the original expression \eqref{eq:kolm} of $P_i$, we obtain $$P_i=\Big(1-\sum\limits_{k=1}^{n+1} x_k^2\Big)\Tilde{f}_ix_i+\sum\limits_{j=1}^{n+1}C_{ij}x_j,$$
where $C_{ij}:=\Big(1-\sum\limits_{k=1}^{n+1} x_k^2\Big)B_{ij}+A_{ij}$. Note that $\deg C_{ij}\leq m-1$. Since $P_i$ is divisible by $x_i$, it follows that $\sum\limits_{j=1}^{n+1}C_{ij}x_j=L_ix_i$ for some $L_i\in \RR[x_1,\ldots,x_{n+1}]$. Thus,
$$0=\sum\limits_{i=1}^{n+1}x_i\sum\limits_{j=1}^{n+1}C_{ij}x_j=\sum\limits_{i=1}^{n+1}L_ix_i^2.$$
Hence, by \Cref{lem:sum-pi-xi-0}, $L_i=\sum\limits_{j=1}^{n+1}\Tilde{A}_{ij}x_j^2$ where $\Tilde{A}_{ij}\in \RR[x_1,\ldots,x_{n+1}]$ with $\deg \Tilde{A}_{ij}\leq m-3$ such that $\Tilde{A}:=(\Tilde{A}_{ij})_{(n+1)\times (n+1)}$ is a skew-symmetric matrix. Therefore,
$$P_i=x_i\Big(\Big(1-\sum\limits_{k=1}^{n+1} x_k^2\Big)\Tilde{f}_i+\sum\limits_{j=1}^{n+1}\Tilde{A}_{ij}x_j^2\Big).$$

One can verify that the converse part is also true with the cofactor $-2\sum\limits_{i=1}^{n+1}\Tilde{f}_ix_i^2$.

\end{proof}

\begin{remark}
A polynomial Kolmogorov vector field on $\SN^n$ has degree at least three.
\end{remark}
\begin{theorem}\label{thm_cubic_firint}
    Suppose that $\chi=(P_1,\ldots,P_{n+1})$ is a cubic Kolmogorov vector field on $\SN^n$ where $P_i$ are given by \eqref{eq:deg-m-form}. If $y=(y_1,\ldots,y_{n+1})\in \RR^{n+1}$ is a non-zero syzygy of $(\Tilde{f}_1,\ldots,\Tilde{f}_{n+1})$ and $\Tilde{A}y=0$ then $H=\prod\limits_{i=1}^{n+1}x_i^{y_i}$ is a first integral of $\chi$.
\end{theorem}
\begin{proof}
    Since $\chi$ is a cubic vector field, $\Tilde{f}_i$ can be chosen as a constant for $i=1,\ldots,n+1$, and $\Tilde{A}$ can be chosen as a constant skew-symmetric matrix. Observe that $x_i=0$ is invariant with cofactor $K_i=\Big(1-\sum\limits_{k=1}^{n+1} x_k^2\Big)  \Tilde{f}_{i} +\sum\limits_{j=1}^{n+1} \Tilde{A}_{ij}x_j^2$. We compute the following.
    \begin{equation*}
            \sum_{i=1}^{n+1}y_iK_i=\Big(1-\sum\limits_{k=1}^{n+1} x_k^2\Big)\Big(\sum_{i=1}^{n+1}y_i\Tilde{f}_i\Big)+\sum_{i=1}^{n+1}y_i\sum_{j=1}^{n+1}\Tilde{A}_{ij}x_j^2=y^t \Tilde{A}\begin{pmatrix}
                x_1^2&\cdots&x_{n+1}^2
            \end{pmatrix}^t=0
    \end{equation*}
    since $\Tilde{A}$ is skew-symmetric. Hence, by Darboux Integrability Theory \cite{Dar78, llibre2018darboux}, $H=\prod\limits_{i=1}^{n+1}x_i^{y_i}$ is a first integral of $\chi$.
\end{proof}

\begin{proposition}\label{thm:linear-fi}
 Let  $f\in \RR[x_1,\ldots,x_{n+1}]$ be a linear polynomial. Then, $f$ is a first integral of some Kolmogorov vector fields on $\SN^n$.
\end{proposition}
\begin{proof}
   Assume that $f:=a_0+\sum\limits_{i=1}^{n+1}a_ix_i$ where $a_j\in \RR$, $0\leq j\leq n+1$. We construct an $(n+1)\times (n+1)$ non-zero skew-symmetric matrix $\Tilde{A}=(\Tilde{A}_{ij})$ such that $a(x)^t\Tilde{A}=0$ where $\Tilde{A}_{ij}\in \RR[x_1,\ldots,x_{n+1}]$ and  $a(x)^t=\begin{pmatrix}
       a_1x_1&\cdots&a_{n+1}x_{n+1}
   \end{pmatrix}$. Note that, by our construction, $(P_1,\ldots,P_{n+1})$ with $P_i=x_i\sum\limits_{j=1}^{n+1}\Tilde{A}_{ij}x_j^2$ is a non-zero vector field on $\SN^n$ such that $f$ is a first integral of it.

Suppose that $a_k\neq 0$ for some $k\in \{1,\ldots,n+1\}$. We take an arbitrary $n\times n$ non-zero skew-symmetric matrix $B:=(B_{ij})$ where $B_{ij}\in \RR[x_1,\ldots,x_{n+1}]$. Now, we define a new skew-symmetric matrix $(A_{ij})=A:=x_kB$. We construct the matrix $\Tilde{A}=(\Tilde{A}_{ij})$ as follows.
\begin{equation*}
\Tilde{A}_{ij}=
    \begin{cases}
        &A_{ij}~\mbox{if } i,j<k,\\
        &A_{i(j-1)}~\mbox{if } i<k,j>k,\\
        &A_{(i-1)j}~\mbox{if } i>k,j<k,\\
        &A_{(i-1)(j-1)}~\mbox{if } i,j>k,\\
    \end{cases}
\end{equation*}
$\Tilde{A}_{kj}=-\frac{1}{a_kx_k}\Big(\sum\limits_{i=1,i\neq k}^{n+1} a_ix_i\Tilde{A}_{ij}\Big)$, and $\Tilde{A}_{jk}=-\Tilde{A}_{kj}$ for $j\neq k$. Moreover, $\Tilde{A}_{kk}=0$. Observe that $\Tilde{A}$ is a non-zero skew-symmetric matrix. For $j\neq k$, we have the following.
\begin{align*}
        \sum\limits_{i=1}^{n+1} a_ix_i\Tilde{A}_{ij}&=\sum\limits_{i=1,i\neq k}^{n+1} a_ix_i\Tilde{A}_{ij}+a_kx_k\Tilde{A}_{kj}\\
        &=\sum\limits_{i=1,i\neq k}^{n+1} a_ix_i\Tilde{A}_{ij}-\sum\limits_{i=1,i\neq k}^{n+1} a_ix_i\Tilde{A}_{ij}\\
        &=0.
\end{align*}
Furthermore, 
\begin{align*}
        \sum\limits_{i=1}^{n+1} a_ix_i\Tilde{A}_{ik}&=\sum\limits_{i=1,i\neq k}^{n+1} a_ix_i\Tilde{A}_{ik}\\
        &=\frac{1}{a_kx_k}\sum\limits_{i=1,i\neq k}^{n+1} a_ix_i \Big(\sum\limits_{j=1,j\neq k}^{n+1} a_jx_j\Tilde{A}_{ji}\Big)\\
        &=\frac{1}{a_kx_k}\Big(\sum\limits_{i=1}^{n+1} a_ix_i \Big(\sum\limits_{j=1}^{n+1} a_jx_j\Tilde{A}_{ji}\Big)-a_kx_k\sum\limits_{j=1,j\neq k}^{n+1}a_jx_j\Tilde{A}_{jk}-a_kx_k\sum\limits_{i=1,i\neq k}^{n+1}a_ix_i\Tilde{A}_{ki}\Big)\\
        &=-\frac{1}{a_kx_k}a(x)^t\Tilde{A}a(x)-\sum\limits_{i=1}^{n+1}a_ix_i(\Tilde{A}_{ik}+\Tilde{A}_{ki})\\
        &=0.
\end{align*}
Therefore, $a(x)$ is orthogonal to the columns of $\Tilde{A}$. Thus, the proof is complete.
\end{proof}

\begin{theorem}\label{thm:cmplt-int-vfld}
There exist completely integrable Kolmogorov vector fields of degree $m$ on $\SN^n$ for each $m\geq 3$ and $n\in \NN$.
\end{theorem}

\begin{proof}
If $\Tilde{A} \in \RR[x_1, \ldots, x_{n+1}]$ is a degree $m-3$ polynomial and $P_1=\Tilde{A}x_1x_2^2, P_2=-\Tilde{A}x_1^2x_2, P_3 = \cdots =P_{n+1}=0$, then $\chi_n:=(P_1, \ldots, P_{n+1})$ is a Kolmogorov vector field on $\SN^n$ of degree $m$. Note that $f_1:=\sum\limits_{i=1}^{n+1}x_i^2-1$ is a first integral of $\chi_n$. Also, $f_j:=x_j$ is a first integral of $\chi_n$ for $j=3,\ldots,n+1$. Hence, $\chi_n$ has $n$ functionally independent first integrals, which makes it completely integrable.
\end{proof}

\begin{proof}[\textbf{Proof of  \Cref{thm_no_Hamilton}}]
Suppose that $(P_1,\ldots,P_{2n})$ is a cubic Kolmogorov vector field on $\SN^{2n-1}$. By \Cref{thm:kolm-vfld}, there exists $\alpha_1,\ldots,\alpha_{2n}\in \RR$ and a constant skew-symmetric matrix $\Tilde{A}=(\Tilde{A}_{ij})$ such that
\begin{equation}\label{eq:ham-form}
    P_i=x_i\Big(\alpha_i \Big(1-\sum\limits_{k=1}^{2n} x_k^2\Big)  +\sum\limits_{j=1}^{2n} \Tilde{A}_{ij}x_j^2\Big).
\end{equation}
We assume that the vector field is Hamiltonian. So, $\frac{\partial P_{2i-1}}{\partial x_{2i-1}}+\frac{\partial P_{2i}}{\partial x_{2i}}=0$ for $1\leq i\leq n$. We compute 
\begin{equation*}
    \frac{\partial P_{2i-1}}{\partial x_{2i-1}}+\frac{\partial P_{2i}}{\partial x_{2i}}=(\alpha_{2i-1}+\alpha_{2i})\Big(1-\sum\limits_{k=1}^{2n} x_k^2\Big)-2(\alpha_{2i-1}x_{2i-1}^2+\alpha_{2i}x_{2i}^2)+\sum\limits_{j=1}^{2n} (\Tilde{A}_{(2i-1)j}+\Tilde{A}_{(2i)j})x_j^2.
\end{equation*}
Therefore, we have the following observations:
\begin{enumerate}[(i)]
    \item $\alpha_{2i-1}+\alpha_{2i}=0$,
    \item $\Tilde{A}_{(2i)(2i-1)}=2\alpha_{2i-1}$,
    \item $\Tilde{A}_{(2i-1)(2i)}=2\alpha_{2i}$,
    \item $\Tilde{A}_{(2i-1)j}+\Tilde{A}_{(2i)j}=0$, for $j\in \{1,\ldots,2i-2,2i+1,\ldots,2n\}$.
\end{enumerate}
We fix $i\in \{1,\ldots,n\}.$ If $H$ is a Hamiltonian for the vector field then $P_{2i-1}=-\frac{\partial H}{\partial x_{2i}}$. Integrating both sides, we obtain
\begin{dmath*}
    H=\alpha_{2i-1}x_{2i-1}\Big(x_{2i}\Big(\sum\limits_{k=1,k\neq 2i}^{2n} x_k^2-1\Big)+\frac{x_{2i}^3}{3}\Big)-x_{2i-1}\Big(\sum\limits_{j=1,j\neq 2i}^{2n} \Tilde{A}_{(2i-1)j}x_j^2 x_{2i}+\Tilde{A}_{(2i-1)(2i)}\frac{x_{2i}^3}{3}\Big)+\phi,
\end{dmath*}
where $\phi$ is a function of $x_1,\ldots,x_{2i-1},x_{2i+1},\ldots,x_{2n}$.
Using the observations $(i)$ and $(iii)$, we simplify $H$ as follows.
\begin{dmath}\label{eq:h-form1}
    H=\alpha_{2i-1}x_{2i-1}x_{2i}\Big(\sum\limits_{k=1}^{2n} x_k^2-1\Big)-x_{2i-1}x_{2i}\Big(\sum\limits_{j=1,j\neq 2i}^{2n} \Tilde{A}_{(2i-1)j}x_j^2 \Big)+\phi.
\end{dmath}
Hence,
$$P_{2i}=\frac{\partial H}{\partial x_{2i-1}}=\alpha_{2i-1}x_{2i}\Big(\sum\limits_{k=1}^{2n} x_k^2-1\Big)+2\alpha_{2i-1}x_{2i-1}^2x_{2i}-x_{2i}\Big(\sum\limits_{j=1,j\neq 2i}^{2n} \Tilde{A}_{(2i-1)j}x_j^2 \Big)+\frac{\partial \phi}{\partial x_{2i-1}}.$$
From \eqref{eq:ham-form}, we obtain
\begin{dmath*}
\alpha_{2i-1}x_{2i}\Big(\sum\limits_{k=1}^{2n} x_k^2-1\Big)+2\alpha_{2i-1}x_{2i-1}^2x_{2i}-x_{2i}\Big(\sum\limits_{j=1,j\neq 2i}^{2n} \Tilde{A}_{(2i-1)j}x_j^2 \Big)+\frac{\partial \phi}{\partial x_{2i-1}}=x_{2i}\Big(\alpha_{2i} \Big(1-\sum\limits_{k=1}^{2n} x_k^2\Big)  +\sum\limits_{j=1}^{2n} \Tilde{A}_{(2i)j}x_j^2\Big).
\end{dmath*}
Note that $\alpha_{2i-1}+\alpha_{2i}=0$. Therefore, 
\begin{align*}
    \frac{\partial \phi}{\partial x_{2i-1}}&=x_{2i}\Big(\sum\limits_{j=1}^{2n} \Tilde{A}_{(2i)j}x_j^2\Big)-2\alpha_{2i-1}x_{2i-1}^2x_{2i}+x_{2i}\Big(\sum\limits_{j=1,j\neq 2i}^{2n} \Tilde{A}_{(2i-1)j}x_j^2 \Big)\\
    &=x_{2i}\Big(-2\alpha_{2i-1}x_{2i-1}^2+\sum\limits_{j=1,j\neq 2i}^{2n} (\Tilde{A}_{(2i)j}+\Tilde{A}_{(2i-1)j})x_j^2\Big)
\end{align*}
Recall that $\phi$ is a function independent of $x_{2i}$. Therefore, $\frac{\partial \phi}{\partial x_{2i-1}}=0$ and hence $\phi$ is independent of $x_{2i-1}$ also.

We take $p\in \{1,\ldots,n\}$ such that $p\neq i$. We have $-\frac{\partial H}{\partial x_{2p}}=P_{2p-1}$. Using the expression of $H$ from \eqref{eq:h-form1}, we obtain
\begin{dmath*}
    \frac{\partial \phi}{\partial x_{2p}}=x_{2p-1}\Big(\alpha_{2p-1}\Big(\sum\limits_{k=1}^{2n} x_k^2-1\Big)-\sum\limits_{j=1}^{2n}\Tilde{A}_{(2p-1)j}x_j^2\Big)+2x_{2i-1}x_{2i}x_{2p}\Big(\Tilde{A}_{(2i-1)2p}-\alpha_{2i-1}\Big).
\end{dmath*}
Recall that $\phi$ is a function independent of $x_{2i-1}$ and $x_{2i}$. Therefore, we have the following observations:
\begin{enumerate}[(a)]
    \item $\Tilde{A}_{(2p-1)(2i-1)}=\alpha_{2p-1}$,
    \item $\Tilde{A}_{(2p-1)(2i)}=\alpha_{2p-1}$,
    \item $\Tilde{A}_{(2i-1)(2p)}=\alpha_{2i-1}$.
\end{enumerate}
Since $i$ and $p$ are arbitrary and $\Tilde{A}_{(2p-1)(2i-1)}+\Tilde{A}_{(2i-1)(2p-1)}=0$, we determine $\alpha_{2p-1}+\alpha_{2i-1}=0$ by observation $(a)$. So, $\alpha_{2q-1}=-\alpha_1$ for $q\in \{2,\ldots,n\}$. Again, $\alpha_3+\alpha_5=0$. Therefore, $\alpha_1=0$ and consequently $\alpha_{2i-1}=0$, for $i\in \{1,\ldots,n\}$. By observation $(i)$, we obtain $\alpha_{2i}=0$. Furthermore, by observations $(ii)$, (a), and (b), we obtain $\Tilde{A}_{rs}=0$ for $r,s\in \{1,\ldots,2n\}$ with $r$ odd. Then, $\Tilde{A}_{rs}=0$ for $r,s\in \{1,\ldots,2n\}$ with $r$ even, by the observation $(iv)$. Therefore, $P_1=\cdots=P_{2n}=0$, which is absurd.
\end{proof}

\begin{proposition}
\label{prop:hom-classify}
The polynomial vector field $(P_1, \ldots, P_{n+1})$ on $\SN^n$ is homogeneous Kolmogorov of degree $m$ if and only if $P_i =x_i\Big( \sum\limits_{j=1}^{n+1} \Tilde{A}_{ij} x_j^2\Big)$ for $i=1, \ldots, n+1$, where $\Tilde{A}_{ij}$ is either zero or a homogeneous polynomial of degree $(m-3)$ such that $(\Tilde{A}_{ij})_{(n+1)\times (n+1)}$ is a skew-symmetric matrix.
\end{proposition}
\begin{proof}
    Suppose that $\chi:=(P_1,\ldots,P_{n+1})$ is a homogeneous polynomial Kolmogorov vector field of degree $m$ on $\SN^n$. Thus, we have $\chi\Big(\sum\limits_{i=1}^{n+1}x_i^2-1\Big)=2\sum\limits_{i=1}^{n+1}P_ix_i=K\Big(\sum\limits_{i=1}^{n+1}x_i^2-1\Big)$ for some $K\in \RR[x_1,\ldots, x_{n+1}]$. Since $P_1,\ldots,P_{n+1}$ are homogeneous polynomials of same degree, $K$ is zero. Moreover, $P_i=\Tilde{P}_ix_i$ for some $\Tilde{P}_i\in \RR[x_1,\ldots,x_{n+1}]$ with $\deg \Tilde{P}_i= m-1$, for $1\leq i\leq n+1$ since $\chi$ is Kolmogorov. Therefore, $\sum\limits_{i=1}^{n+1}\Tilde{P}_ix_i^2=0$. By \Cref{lem:sum-pi-xi-0}, $\Tilde{P}_i=\sum\limits_{j=1}^{n+1}\Tilde{A}_{ij}x_j^2$ where  $\Tilde{A}_{ij}$ is either zero or a homogeneous polynomial of degree $(m-3)$ and $(\Tilde{A}_{ij})$ is a skew-symmetric matrix.

    The converse part also holds.
\end{proof}

Next, we discuss some necessary conditions for the existence of invariant great $(n-1)$-spheres of a polynomial vector field on $\SN^n$.

\begin{proposition}\cite[Lemma 5.2]{jana2024dynamics}\label{lem:cone-invariant}
    Let $\mathcal{S}:=\Big\{\sum\limits_{i=1}^{n+1} a_ix_i=d\Big\}\cap \SN^n$ be an $(n-1)$-sphere in $\SN^n$ and $$C(\mathcal{S}):=\Big\{(x_1,\dots,x_{n+1})\in \RR^{n+1} ~|~\Big(\sum\limits_{i=1}^{n+1} a_ix_i\Big)^2-d^2\sum\limits_{i=1}^{n+1} x_i^2=0\Big\}.$$ 
   Suppose that $\chi$ is a homogeneous polynomial vector field on $\SN^n$. Then, $\mathcal{S}$ is invariant if and only if $\mathcal{C}(\mathcal{S})$ is invariant for the vector field $\chi$.
\end{proposition}

\begin{theorem}\label{thm_sp_grtsp}
Suppose that $\chi=(P_1,\dots,P_{n+1})$ is a homogeneous polynomial Kolmogorov vector field on $\SN^n$ of degree at most four. If $\{x_{n+1}+d=0\}\cap \SN^n$ is an invariant $(n-1)$-sphere, it must be a great $(n-1)$-sphere.
\end{theorem}

\begin{proof}
   We assume that $\chi$ is of the form given in \Cref{prop:hom-classify}. Suppose that $\{x_{n+1}+d=0\}\cap \SN^n$ is an invariant $(n-1)$-sphere for $\chi$. We need to prove that $d=0$.
    By \Cref{lem:cone-invariant}, $\mathcal{S}=\{x_{n+1}+d=0\}\cap \SN^n$ is invariant if and only if $\mathcal{C}(\mathcal{S})=\Big\{x_{n+1}^2-d^2\sum\limits_{i=1}^{n+1}x_i^2=0\Big\}$ is invariant. If $\mathcal{C}(\mathcal{S})$ is invariant, there exists $K\in \RR[x_1,\ldots,x_{n+1}]$ such that
    \begin{align*}
        &\chi\Big(x_{n+1}^2-d^2\sum\limits_{i=1}^{n+1}x_i^2\Big)=K\Big(x_{n+1}^2-d^2\sum\limits_{i=1}^{n+1}x_i^2\Big) \\
        \implies& 2P_{n+1} x_{n+1}=K\Big(x_{n+1}^2-d^2\sum\limits_{i=1}^{n+1}x_i^2\Big)~\mbox{since $\chi$ is homogeneous}.\\
    \implies & 2x_{n+1}^2\Big(\sum_{j=1}^n \Tilde{A}_{(n+1)j}x_j^2\Big)=K\Big(x_{n+1}^2-d^2\sum\limits_{i=1}^{n+1}x_i^2\Big).
    \end{align*}
 Assume that $d\neq 0$. Then $x_{n+1}^2$ divides $K$. Note that $\deg K\leq 3$ since $\deg \chi\leq 4$. So, 
 \begin{equation}\label{d=0}
     \sum_{j=1}^n \Tilde{A}_{(n+1)j}x_j^2=\widehat{K}\Big(x_{n+1}^2-d^2\sum\limits_{i=1}^{n+1}x_i^2\Big),
 \end{equation}
for some polynomial $\widehat{K}$ of degree at most one. If $\widehat{K}\neq 0$, then there is a monomial $\alpha x_{n+1}^2$ or $\alpha x_jx_{n+1}^2$ in the right-hand side of \eqref{d=0} for some $\alpha\in \RR\setminus \{0\}$ and $j\in \{1,\ldots,n+1\}$. But, such a monomial does not exist in the left-hand side of \eqref{d=0} since $\deg \Tilde{A}_{(n+1)j}\leq 1$ for $j=1,\ldots,n$. So, $\widehat{K}$ is zero and consequently $P_{n+1}=0$. This contradicts the fact that $\chi$ is homogeneous. Hence, $d$ is zero.
\end{proof}

\begin{proof}[\textbf{Proof of  \Cref{thm_inv_grt_deg1}}]
Suppose that $\chi$ is of the form given in \Cref{prop:hom-classify}. Since $\chi$ is cubic, $P_i=x_ix^t\Tilde{A}_ix$ where $x^t=(x_1~\ldots~x_{n+1})$ and $\Tilde{A}_i={\rm diag}(\Tilde{A}_{i1},\ldots,\Tilde{A}_{i(n+1)})$ is a constant matrix. If $\sum\limits_{i=1}^{n+1}a_ix_i=0$ is invariant, then there exists a homogeneous quadratic polynomial $K$ such that
\begin{equation}\label{eq:eq1-inv-sphere}
\sum\limits_{i=1}^{n+1}a_iP_i=K\Big(\sum\limits_{i=1}^{n+1}a_ix_i\Big).    
\end{equation}
Moreover, there exists a constant symmetric matrix $B$ of order $(n+1)$ such that $K=x^t Bx$. Therefore, we have
$$\sum\limits_{i=1}^{n+1}a_ix_ix^t \Tilde{A}_ix=(x^tBx)\Big(\sum\limits_{i=1}^{n+1}a_ix_i\Big).$$
This simplifies to 
\begin{equation}\label{eq:deg3-hom-inv-sphere}
    x^t \Big(\sum\limits_{i=1}^{n+1} a_ix_i(\Tilde{A}_i-B)\Big)x=0.
\end{equation}
If $x_1=\cdots=x_{n+1}=1$, then the above equality represents that the sum of all the elements of the matrix $\sum\limits_{i=1}^{n+1}a_i(\Tilde{A}_i-B)=0$. Consequently, sum of all the elements of $\sum\limits_{i=1}^{n+1}a_i\Tilde{A}_i=$ $(a_1+\cdots+a_{n+1})\times$(sum of all the elements of $B$). Observe that $$a_i\Tilde{A}_i={\rm diag}(a_i\Tilde{A}_{i1},\ldots,a_i\Tilde{A}_{i(n+1)}).$$ Therefore, $\sum\limits_{i,j=1}^{n+1} a_i\Tilde{A}_{ij}=(a_1+\cdots+a_{n+1})\times$(sum of all the elements of $B$). This proves $(i)$.

Recall that $P_i=x_i\Big(\sum\limits_{j=1}^{n+1}\Tilde{A}_{ij}x_j^2\Big)$. If $x_i=a_i$ for $1\leq i\leq n+1$, then \eqref{eq:eq1-inv-sphere} gives
\begin{align*}
        &\sum\limits_{i=1}^{n+1} a_i^2 \sum\limits_{j=1}^{n+1}\Tilde{A}_{ij}a_j^2=(a_1^2+\cdots+a_{n+1}^2)K(a_1,\ldots,a_{n+1})\\
    \implies & (a_1^2~\cdots~a_{n+1}^2) \Tilde{A} (a_1^2~\cdots~a_{n+1}^2)^t=(a_1^2+\cdots+a_{n+1}^2)K(a_1,\ldots,a_{n+1}),
\end{align*}
where $\Tilde{A}=(\Tilde{A}_{ij})$. Since $\Tilde{A}$ is a skew-symmetric matrix, we obtain $K(a_1,\ldots,a_{n+1})=0$.
\end{proof}

\section{Integrability of cubic vector fields}\label{sec:inv_mer_par}
In this section,  we examine when a cubic Kolmogorov vector field on $\SN^n$ has a Darboux first integral. In several cases, our observations are necessary and sufficient.

We discuss invariant hyperplanes $a_0+\sum\limits_{i=1}^{n+1}a_ix_i=0$ of the cubic vector fields given by \eqref{eq:deg-m-form}. Observe that $x_i=0$ is invariant with a cofactor of the form $k_0+\sum\limits_{i=1}^{n+1} k_ix_i^2$. Next result investigates invariant hyperplanes with at least two $a_j$ non-zero for $0\leq j\leq n+1$.

\begin{theorem}\label{thm_hypln_cofac}
    Suppose that $\chi$ is a cubic Kolmogorov vector field on $\SN^n$ of the form \eqref{eq:deg-m-form}. Then the hyperplane $a_0+\sum\limits_{i=1}^{n+1}a_ix_i=0$ with at least two $a_j$ non-zero for $0\leq j\leq n+1$ is invariant with cofactor $k_0+\sum\limits_{i=1}^{n+1}k_ix_i^2$ if and only if one of the following holds.
    \begin{enumerate}[(i)]
        \item $a_0\neq 0,~k_0=0,~a_i\alpha_i=0$, $a_i\Tilde{A}_{ij}=0$, and $k_i=0$ for $1\leq i,j\leq n+1$.
        \item $a_0=0$. If $a_{i_1},a_{i_2}\neq 0$ then $\alpha_{i_1}=\alpha_{i_2}=k_0,~\Tilde{A}_{(i_1)(i_2)}=0,~\Tilde{A}_{(i_1)j}=\Tilde{A}_{(i_2)j}$, and  $k_j=\Tilde{A}_{(i_1)j}-k_0$, for $1\leq j\leq n+1$.
    \end{enumerate}
\end{theorem}
\begin{proof}
   Suppose that the hyperplane $a_0+\sum\limits_{i=1}^{n+1}a_ix_i=0$ is invariant. So,
    \begin{equation}\label{plane-inv}
        \sum_{i=1}^{n+1}a_ix_i \Big(\Big(1-\sum_{k=1}^{n+1}x_k^2\Big)\alpha_i+\sum_{j=1}^{n+1}\Tilde{A}_{ij}x_j^2 \Big)=\Big(k_0+ \sum_{i=1}^{n+1}k_ix_i^2 \Big)\Big(a_0+\sum_{i=1}^{n+1}a_ix_i \Big).
    \end{equation}
    Comparing the degree three, degree two, degree one, and constant terms in \eqref{plane-inv}, we obtain
    \begin{equation}\label{degree-3-mon}
        \sum_{i=1}^{n+1}a_ix_i\Big(-\alpha_i\sum_{k=1}^{n+1}x_k^2+\sum_{j=1}^{n+1}\Tilde{A}_{ij}x_j^2\Big)=\Big(\sum_{i=1}^{n+1}k_ix_i^2\Big)\Big(\sum_{i=1}^{n+1}a_ix_i\Big),
    \end{equation}
        \begin{equation}\label{degree-2-mon}
        0=a_0\Big(\sum_{i=1}^{n+1}k_ix_i^2\Big),
    \end{equation}
        \begin{equation}\label{degree-1-mon}
        \sum_{i=1}^{n+1}a_i \alpha_i x_i=k_0\sum_{i=1}^{n+1}a_ix_i,~\mbox{and}
    \end{equation}
        \begin{equation}\label{constant-term}
        0=k_0a_0,
    \end{equation}
respectively.

\noindent\textbf{Case-1:} $a_0\neq 0$. Then $k_0$ must be zero. Also, from \eqref{degree-2-mon} and \eqref{degree-1-mon}, we obtain $k_i=0$ and $a_i\alpha_i=0$, respectively, for $1\leq i\leq n+1$. Hence, from \eqref{degree-3-mon}, we have $\sum\limits_{i=1}^{n+1}a_ix_i(\sum\limits_{j=1}^{n+1}\Tilde{A}_{ij}x_j^2)=0.$ Therefore, $a_i\Tilde{A}_{ij}=0$ for $1\leq i,j\leq n+1$.

\noindent\textbf{Case-2:} $a_0=0$. Observe that \eqref{degree-1-mon} implies $a_i\alpha_i=k_0a_i$ for $1\leq i\leq n+1$. Hence, rewriting \eqref{degree-3-mon}, we get
\begin{equation}\label{rewrite-degree-3-mon}
    \sum_{i=1}^{n+1}\sum_{j=1}^{n+1}a_i\Tilde{A}_{ij}x_ix_j^2=\Big(\sum_{i=1}^{n+1}(k_i+k_0)x_i^2 \Big) \Big(\sum_{i=1}^{n+1}a_ix_i \Big).
\end{equation}
So, $\sum\limits_{i=1}^{n+1}\sum\limits_{j=1}^{n+1}a_i(\Tilde{A}_{ij}-k_j-k_0)x_ix_j^2=0.$ Consequently, 
\begin{equation}\label{a-i-prod}
    a_i(\Tilde{A}_{ij}-k_j-k_0)=0,
\end{equation}
for $1\leq i,j\leq n+1$. Suppose that $a_{i_1},a_{i_2}\neq 0$ for some $i_1,i_2\in \{1,\ldots,n+1\}$ with $i_1\neq i_2$. Then, it follows from \eqref{degree-1-mon} that $\alpha_{i_1}=\alpha_{i_2}=k_0$. By \eqref{a-i-prod}, $k_j=\Tilde{A}_{(i_1)j}-k_0$ and $\Tilde{A}_{(i_1)j}=\Tilde{A}_{(i_2)j}$. Moreover,
$$\Tilde{A}_{(i_1)(i_2)}-k_{i_2}-k_0=0~\mbox{and}~\Tilde{A}_{(i_2)(i_1)}-k_{i_1}-k_0=0.$$ In \eqref{rewrite-degree-3-mon}, comparing the coefficients of $x_{i_1}^3$ and $x_{i_2}^3$, together with the fact that $(\Tilde{A}_{ij})$ is a skew-symmetric matrix, we get $k_{i_1}=k_{i_2}=-k_0$. Therefore, $\Tilde{A}_{(i_1)(i_2)}=\Tilde{A}_{(i_2)(i_1)}=0$.

One can check that the converse part is also true.
\end{proof}

\begin{theorem}\label{first-integral-ns}
    Suppose that $\chi=(P_1,\ldots,P_{n+1})$ is a cubic Kolmogorov vector field on $\SN^n$ such that
    $$P_i=x_i\Big(\alpha_i\Big(1-\sum_{k=1}^{n+1}x_k^2\Big)+\sum_{j=1}^{n+1}\Tilde{A}_{ij}x_j^2\Big)$$ where
    $\alpha_i\in \RR$ for $i=1,\ldots,n+1$ and $\Tilde{A}=(\Tilde{A}_{ij})$ is a constant skew-symmetric matrix. If $\chi$ has an invariant algebraic hypersurface $g_{n+2}=0$ other than $x_i=0$ for $i=1,\ldots,n+1$ with cofactor $K_{n+2}=k_0+\sum\limits_{i=1}^{n+1}k_ix_i^2$ then $\chi$ has the Darboux first integral $$H:=g_{n+2}^{\beta_{n+2}}\prod_{i=1}^{n+1}x_i^{\beta_i}$$ if and only if the matrix
    \begin{equation}\label{eq:matrix-b}
        B=\begin{pmatrix}
        \alpha_1&-\alpha_1&\Tilde{A}_{12}-\alpha_1&\cdots&\Tilde{A}_{1(n+1)}-\alpha_1\\
        \alpha_2&\Tilde{A}_{21}-\alpha_2&-\alpha_2&\cdots&\Tilde{A}_{2(n+1)}-\alpha_2\\
        \vdots&\vdots&\vdots&\vdots&\vdots\\
        \alpha_{n+1}&\Tilde{A}_{(n+1)1}-\alpha_{n+1}&\Tilde{A}_{(n+1)2}-\alpha_{n+1}&\cdots&-\alpha_{n+1}\\
        k_0&k_1&k_2&\cdots&k_{n+1}
    \end{pmatrix}_{(n+2)\times (n+2)}
    \end{equation}
    has rank less than $(n+2)$.
\end{theorem}
\begin{proof}
The hyperplane $x_i=0$ is invariant with the cofactor $K_i=\alpha_i+\sum\limits_{j=1}^{n+1}(\Tilde{A}_{ij}-\alpha_i) x_j^2$ where $\Tilde{A}_{ii}=0$. By Darboux Integrability Theory \cite{Dar78, DLA06}, $H=g_{n+2}^{\beta_{n+2}}\prod\limits_{i=1}^{n+1}x_i^{\beta_i}$ is a first integral of $\chi$ if and only if 
\begin{equation}\label{sum-beta-i-zero}
    \beta_{n+2}K_{n+2}+\sum_{i=1}^{n+1}\beta_i K_i=0,
\end{equation}
where not all $\beta_j$ is zero for $j\in \{1,\ldots, n+2\}$. Next, we compare the coefficients of each monomial in \eqref{sum-beta-i-zero}.
\begin{align*}
 \mbox{Constant-term}:&~\beta_1\alpha_1+\beta_2\alpha_2+\cdots+\beta_{n+1}\alpha_{n+1}+\beta_{n+2}k_0=0\\
    x_1^2:&~\beta_1(-\alpha_1)+\beta_2(\Tilde{A}_{21}-\alpha_2)+\cdots+\beta_{n+1}(\Tilde{A}_{(n+1)1}-\alpha_{n+1})+\beta_{n+2}k_1=0\\
    x_2^2:&~\beta_1(\Tilde{A}_{12}-\alpha_1)+\beta_2(-\alpha_2)+\cdots+\beta_{n+1}(\Tilde{A}_{(n+1)2}-\alpha_{n+1})+\beta_{n+2}k_2=0\\
    &\hspace{50mm}\vdots\\
    x_{n+1}^2:&~\beta_1(\Tilde{A}_{1(n+1)}-\alpha_1)+\beta_2(\Tilde{A}_{2(n+1)}-\alpha_2)+\cdots+\beta_{n+1}(-\alpha_{n+1})+\beta_{n+2}k_{n+1}=0.
\end{align*}
Hence, \eqref{sum-beta-i-zero} holds if and only if $\begin{pmatrix}
    \beta_1&\beta_2&\cdots&\beta_{n+2}
\end{pmatrix}$ is a non-trivial solution of the matrix equation $yB=0$. So, $H$ is a first integral if and only if ${\rm rank}(B)<(n+2)$.
\end{proof}

We remark that if $\chi$ is a homogeneous polynomial vector field on $\SN^n$, then any sphere $g_r:=\sum\limits_{i=1}^{n+1}x_i^2-r^2=0$ is a first integral of $\chi$, see \cite[Theorem 5.4]{jana2024dynamics}. The next result proves the converse part when $\chi$ is a cubic Kolmogorov vector field on $\SN^n$.

\begin{proposition}\label{prop_fint1}
Let $\chi$ be a cubic Kolmogorov vector field on $\SN^n$. If $g_r:=\sum\limits_{i=1}^{n+1}x_i^2-r^2=0$ with $r\neq 0,\pm 1$ is an invariant sphere for $\chi$, then the vector field is homogeneous and $g_r$ is a first integral.
\end{proposition}
\begin{proof}
If $\chi=(P_1, \ldots, P_{n+1})$ is a cubic Kolmogorov vector field on $\SN^n$ then $$P_i = x_i \Big( \Big(1- \sum_{k=1}^{n+1} x^2_k\Big) \alpha_i + \sum_{j=1}^{n+1} \Tilde{A}_{ij} x^2_j \Big), $$
where $ (\Tilde{A}_{ij} ) $ is a constant skew-symmetric matrix  of  order $(n+1)$ and $\alpha_i \in \RR$ for $i=1, \ldots, n+1$. Suppose that $g_r=0$ is invariant for $\chi$. Then,
\begin{equation}\label{second-sphere}
    \Big(1-\sum_{k=1}^{n+1} x^2_k\Big)\sum_{i=1}^{n+1} \alpha_i x_i^2 + \sum_{i, j=1}^{n+1}\Tilde{A}_{ij} x_i^2x_j^2 = K \Big(\sum_{i=1}^{n+1}x_i^2-r^2\Big)
\end{equation}
for some $K \in \RR[x_1,\ldots,x_{n+1}]$ with $\deg K\leq 2$. Each monomial in the left-hand side of \eqref{second-sphere} has degree at least two. Hence, either $K$ is zero or a quadratic homogeneous polynomial. Observe that $\sum\limits_{i, j=1}^{n+1}\Tilde{A}_{ij} x_i^2x_j^2=0$ since $(\Tilde{A}_{ij})$ is a skew-symmetric matrix. Hence, $ \Big(1-\sum\limits_{k=1}^{n+1} x^2_k\Big)$ divides $\Big(\sum\limits_{i=1}^{n+1}x_i^2-r^2\Big)$ if $K$ is quadratic homogeneous. This contradicts the assumption that $r\neq 1$. So, $K$ must be zero. Consequently, $\alpha_i=0$ for $i=1,\ldots, n+1$. Therefore, $\chi$ is homogeneous and $g_r$ is a first integral.
\end{proof}

\begin{proof}[\textbf{Proof of \Cref{thm_indp_fi}}]
    Suppose that ${\rm rank}(B)\leq 2$. Then, the matrix equation $yB=0$ has at least $n$ independent solutions, say $y_1,\ldots,y_n$ where $y_i=\begin{pmatrix}
        y_{i1}&y_{i2}&\cdots&y_{i(n+2)}
    \end{pmatrix}$ for $i=1,\ldots,n$. From the proof of \Cref{first-integral-ns}, we recall that if $y$ is solution of $yB=0$ then $\langle y,(K_1,K_2,\ldots,K_{n+2})\rangle=0$ where $K_i$ is the cofactor of $x_i=0$ for $i=1,\ldots,n+1$ and $K_{n+2}$ is the cofactor of $g_{n+2}=0$. Hence,
    $$\sum_{j=1}^{n+2} y_{ij}K_j=0~\mbox{for}~1\leq i\leq n.$$
    It follows that $H_i=g_{n+2}^{y_{i(n+2)}}\prod\limits_{j=1}^{n+1}x_j^{y_{ij}}$ is a first integral of $\chi$ for $i=1,\ldots,n$. We claim that $H_1,\ldots, H_n$ are functionally independent.

    We compute $\frac{\partial H_i}{\partial x_k}=(\frac{g_{n+2}}{x_k} y_{ik}+ y_{i(n+2)}\frac{\partial g_{n+2}}{\partial x_k})g_{n+2}^{y_{i(n+2)}-1}\prod\limits_{j=1}^{n+1}x_j^{y_{ij}}$. We define $Y:=(y_{ij})_{n\times (n+2)}$ such that $\bf y_1,\ldots ,\bf y_{n+2}$ are the columns of $Y$. Suppose that $J:=\frac{\partial(H_1,\ldots,H_n)}{\partial (x_1,\ldots,x_{n+1})}$ having columns $J_1,\ldots,J_{n+1}$. Then $$J_p={\rm diag} (g_{n+2}^{y_{1(n+2)}-1}\prod\limits_{j=1}^{n+1}x_j^{y_{1j}},\ldots,g_{n+2}^{y_{n(n+2)}-1}\prod\limits_{j=1}^{n+1}x_j^{y_{nj}})  (\frac{g_{n+2}}{x_p}{\bf y_p} +\frac{\partial g_{n+2}}{\partial x_p}\bf y_{n+2}),$$ for $1\leq p\leq n+1$. We have to prove that ${\rm rank}(J)$ is $n$ at some point. Consider the matrix $\Tilde{J}$ having columns $\Tilde{J}_1,\ldots, \Tilde{J}_n$ where $\Tilde{J}_p:=\frac{g_{n+2}}{x_p}{\bf y_p} +\frac{\partial g_{n+2}}{\partial x_p}\bf y_{n+2}$ where $1\leq p\leq n$. Note that ${\rm diag} (g_{n+2}^{y_{1(n+2)}-1}\prod\limits_{j=1}^{n+1}x_j^{y_{1j}},\ldots,g_{n+2}^{y_{n(n+2)}-1}\prod\limits_{j=1}^{n+1}x_j^{y_{nj}})$ is an invertible matrix at each $z_{ij}$ for $i,j\in \{1,\ldots,n+1\}$. Therefore, ${\rm rank}(J)={\rm rank}(\Tilde{J})$ at each $z_{ij}$. We prove that ${\rm rank}(\Tilde{J})=n$ at some $z_{ij}$.
    
    If ${\bf y_{n+2}}=0$ then ${\rm rank}(\Tilde{J})=n$ at each $z_{ij}$ since ${\rm rank}(Y)=n$. If ${\bf y_{n+2}}\neq 0$ then without loss of generality, we assume that ${\bf y_3,\ldots, y_{n+2}}$ are independent column vectors. Then, $\Tilde{J}_3,\ldots,\Tilde{J}_{n+1}$ are independent columns vectors at each $z_{ij}$. We claim that $\Tilde{J}_2,\ldots,\Tilde{J}_{n+1}$ are independent at $z_{1j}$, for $1\leq j\leq n+1$. Similarly, we can also prove $\Tilde{J}_1, \Tilde{J}_3,\ldots,\Tilde{J}_{n+1}$ are independent at $z_{2j}$, for $1\leq j\leq n+1$. Let $z_{1j}:=(s^j_{11},\ldots,s^j_{1(n+1)})$.

If ${\bf y_2}=0$ then $\sum\limits_{i=2}^{n+1} c_i\Tilde{J}_i(z_{1j})=0$ gives $g_{n+2}(z_{1j})\sum\limits_{k=3}^{n+1} \frac{c_k}{s^j_{1k}}{\bf y_k}+(\sum\limits_{k=2}^{n+1} c_k\frac{\partial g_{n+2}}{\partial x_k}(z_{1j})){\bf y_{n+2}}=0$. Hence, $c_k=0$ for $k=2,\ldots,n+1$. Therefore, $\Tilde{J}_2,\ldots,\Tilde{J}_{n+1}$ are independent at $z_{1j}$.

Now, consider ${\bf y_2}\neq 0$. Suppose that there exist $d_2,\ldots,d_{n+1}\in \RR$, not all zero, such that
\begin{align*}
        &\sum\limits_{k=2}^{n+1} d_k\Tilde{J}_k(z_{1j})=0\\
      \implies& g_{n+2}(z_{1j})\sum\limits_{k=2}^{n+1} \frac{d_k}{s^j_{1k}}{\bf y_k}+(\sum\limits_{k=2}^{n+1} d_k\frac{\partial g_{n+2}}{\partial x_k}(z_{1j})){\bf y_{n+2}}=0 \\
    \implies& \begin{pmatrix}
        {\bf y_2}&\ldots&{\bf y_{n+2}}
    \end{pmatrix}\begin{pmatrix}
        \frac{d_2}{s^j_{12}}g_{n+2}(z_{1j})&\ldots&\frac{d_{n+1}}{s^j_{1(n+1)}}g_{n+2}(z_{1j})&\sum\limits_{k=2}^{n+1} d_k\frac{\partial g_{n+2}}{\partial x_k}(z_{1j})
    \end{pmatrix}^t=0.
\end{align*}
Rank of the matrix $\begin{pmatrix}
        {\bf y_2}&\ldots&{\bf y_{n+2}}
    \end{pmatrix}$ is $n$. So, the nullity of it is 1, say $\begin{pmatrix}
        p_2&\ldots&p_{n+2}
    \end{pmatrix}^t$ is a basis vector of the null space. Therefore,
    $$\begin{pmatrix}
        \frac{d_2}{s^j_{12}}g_{n+2}(z_{1j})&\ldots&\frac{d_{n+1}}{s^j_{1(n+1)}}g_{n+2}(z_{1j})&\sum\limits_{k=2}^{n+1} d_k\frac{\partial g_{n+2}}{\partial x_k}(z_{1j})
    \end{pmatrix}^t=\alpha \begin{pmatrix}
        p_2&\ldots&p_{n+2}
    \end{pmatrix}^t,$$
    for some $\alpha\in \RR\setminus \{0\}$. Hence, we obtain $p_{n+2}=\frac{1}{g_{n+2}(z_{1j})}(\sum\limits_{k=2}^{n+1}p_ks^j_{1k}\frac{\partial g_{n+2}}{\partial x_k}(z_{1j}))$. Thus, 
    $$\begin{pmatrix}
        s^j_{12}\frac{\partial g_{n+2}}{\partial x_1}(z_{1j}) &\ldots&s^j_{1(n+1)}\frac{\partial g_{n+2}}{\partial x_{n+1}}(z_{1j})&-g_{n+2}(z_{1j})
    \end{pmatrix}^t$$ is a null space solution of $\begin{pmatrix}
        p_2&\ldots&p_{n+2}
    \end{pmatrix}$, for $1\leq j\leq n+1$. By the hypothesis, $\begin{pmatrix}
        p_2&\ldots&p_{n+2}
    \end{pmatrix}$ has $n+1$ independent null space solutions, which is impossible. So, our assumption is wrong. Therefore, $\Tilde{J}_2,\ldots,\Tilde{J}_{n+1}$ are independent at $z_{1j}$, for $1\leq j\leq n+1$. Hence, ${\rm rank}(\Tilde{J})=n$ at $z_{1j}$.

    Now, we assume that $H_i$ are functionally independent first integrals of $\chi$. By Darboux Integrability Theory \cite{Dar78, DLA06}, $y_i$ is a non-trivial solution of $yB=0$ where $y_i=\begin{pmatrix}
        y_{i1}&y_{i2}&\cdots&y_{i(n+2)}
    \end{pmatrix}$ for $i=1,\ldots,n$. We claim that $y_1,\ldots,y_n$ are linearly independent. If not, assume that $\sum\limits_{i=1}^n c_iy_i=0$ where not all $c_i$ are zero. Hence, $\prod\limits_{i=1}^n H_i^{c_i}-1=0$. This implies $H_1,\ldots, H_n$ are functionally dependent, which is a contradiction. Hence, ${\rm rank}(B)\leq 2$. Thus, the proof is complete.
\end{proof}

\begin{remark}
    By the hypothesis of \Cref{thm_indp_fi}, $z_{ij}$ must be distinct for any fixed $i$ when $j$ varies. However, $z_{i_1j_1}$ can be equal to $z_{i_2j_2}$ for some $i_1,i_2,j_1,j_2\in \{1,\ldots,n+1\}$ with $i_1\neq i_2$.
\end{remark}
\begin{corollary}\label{cor:ind-fi}
Suppose that $g_{n+2}:=\sum\limits_{i=1}^{n+1} x_i^2-1$ in \Cref{first-integral-ns}, and the rank of the matrix $B$ in \eqref{eq:matrix-b} is less than or equal to 2 for the cubic vector field \eqref{eq:deg-m-form}. Then the vector field is completely integrable in $\RR^{n+1}\setminus \{x_1 \cdots x_{n+1}=0\}$.
\end{corollary}
\begin{proof}
The $n$-sphere given by $g_{n+2}=0$ is invariant under the cubic vector field \eqref{eq:deg-m-form} with cofactor $-2 \sum\limits_{i=1}^{n+1} \alpha_i x_i^2$. So, $k_0=0$ and $k_i=-2\alpha_i$ for $i\in \{1,\ldots, n+1\}$ in matrix $B$. We choose $z_1,\ldots,z_{n+1}\in \RR^{n+1}\setminus \SN^{n}$, where the $j$-th coordinate of $z_j$ is 2 and all other coordinates are 1, for each ${j\in \{1,\ldots,n+1\}}$. For each $i\in \{1,\ldots, n+1\}$, evaluating the vector
    $$\begin{pmatrix}
        x_1\frac{\partial g_{n+2}}{\partial x_1} & \ldots&x_{i-1}\frac{\partial g_{n+2}}{\partial x_{i-1}}&x_{i+1}\frac{\partial g_{n+2}}{\partial x_{i+1}}&\ldots&x_{n+1}\frac{\partial g_{n+2}}{\partial x_{n+1}}&-g_{n+2}
    \end{pmatrix}$$
 at the points $z_1,\ldots,z_{n+1}$, we obtain the row vectors of the matrix
$$M=\begin{pmatrix}
    8&2&2&\cdots&2&-n-3\\
    2&8&2&\cdots&2&-n-3\\
    2&2&8&\cdots&2&-n-3\\
    \vdots&\vdots&\vdots&\vdots&\vdots&\vdots\\
    2&2&2&\cdots&8&-n-3\\
    2&2&2&\cdots&2&-n-3
\end{pmatrix}_{(n+1)\times (n+1)}.$$
We claim that $\det (M)\neq 0$. After subtracting the $(n+1)$-th row from the first row, we have $$\det (M)=\begin{vmatrix}
    6&0&0&\cdots&0&0\\
    2&8&2&\cdots&2&-n-3\\
    2&2&8&\cdots&2&-n-3\\
    \vdots&\vdots&\vdots&\vdots&\vdots&\vdots\\
    2&2&2&\cdots&8&-n-3\\
    2&2&2&\cdots&2&-n-3
\end{vmatrix}_{(n+1)\times (n+1)}.$$
Therefore, $\det (M)=6 \begin{vmatrix}
    8&2&2&\cdots&2&-n-3\\
    2&8&2&\cdots&2&-n-3\\
    2&2&8&\cdots&2&-n-3\\
    \vdots&\vdots&\vdots&\vdots&\vdots&\vdots\\
    2&2&2&\cdots&8&-n-3\\
    2&2&2&\cdots&2&-n-3
\end{vmatrix}_{n\times n}$, and inductively, $$\det (M)=6^{n-1}\begin{vmatrix}
    8&-n-3\\
    2&-n-3
\end{vmatrix}=-6^n(n+3).$$
Hence, by \Cref{thm_indp_fi}, the vector field is completely integrable in the open set $\RR^{n+1}\setminus \{x_1\cdots x_{n+1}=0\}$.
\end{proof}

\begin{example}
    Consider the vector field $\chi=(P_1,P_2,P_3)$ on $\SN^2$ defined by 
    \begin{align*}
            P_1&=x_1\Big(\beta\Big(1-\sum\limits_{i=1}^3 x_i^2\Big)+\alpha x_2^2\Big),\\
            P_2&=-\alpha x_2(x_1^2+x_3^2),\\
            P_3&=x_3\Big(\beta\Big(1-\sum\limits_{i=1}^3 x_i^2\Big)+\alpha x_2^2\Big),
    \end{align*}
    where $\alpha,\beta \in \RR\setminus \{0\}$. Assume that $g_4:=x_1^2+x_2^2+x_3^2-1$. We compute the associated matrix
$$B=\begin{bmatrix}
    \beta&-\beta&\alpha-\beta&-\beta\\
    0&-\alpha&0&-\alpha\\
    \beta&-\beta&\alpha-\beta&-\beta\\
    0&-2\beta&0&-2\beta\\
\end{bmatrix}.$$ Observe that ${\rm rank}(B)=2$. Moreover, $\begin{pmatrix}
    1& 0& -1& 0
\end{pmatrix}$ and $\begin{pmatrix}
    0& -2\beta& 0& \alpha
\end{pmatrix}$ are two linearly independent solutions of $yB=0$. Therefore, the functions $x_1x_3^{-1}$ and $g_4^{\alpha}x_2^{-2\beta}$ are two functionally independent first integrals of $\chi$.

On the other hand, consider the points $z_1=(2,1,1,)$, $z_2=(1,2,1)$, and $z_3=(1,1,2)$. Then, for each $i\in \{1,2,3\}$, the vectors
$$\begin{pmatrix}
        x_1\frac{\partial g_4}{\partial x_1} & \ldots&x_{i-1}\frac{\partial g_4}{\partial x_{i-1}}&x_{i+1}\frac{\partial g_4}{\partial x_{i+1}}&\ldots&x_4\frac{\partial g_4}{\partial x_4}&-g_4
    \end{pmatrix}$$
evaluated at $z_1,z_2,z_3$ are linearly independent (See the proof of \Cref{cor:ind-fi}). Hence, the hypothesis of \Cref{thm_indp_fi} is satisfied.
\end{example}

\noindent {\bf Acknowledgments.}
The first author was supported by the Prime Minister's Research Fellowship, Government of India (Grant No. SB22230339MAPMRF008771). The second author thanks the `ICSR office at IIT Madras' for the SEED research grant (Grant No. IP21220071MANFSC008771) and SERB (Now ANRF) India for the CRG Grant (CRG/2023/000239).


\end{document}